\documentclass[12pt]{article}
\usepackage{amsmath}
\usepackage{algorithmic}
\usepackage[Algorithm,ruled]{algorithm}
\usepackage{titlesec}
\usepackage{sectsty}
\usepackage[left=2cm,right=2cm,top=2cm,bottom=2cm]{geometry}
\usepackage{amssymb,amscd}
\usepackage{amsthm}
\usepackage{graphicx}
\usepackage[normalem]{ulem}
\usepackage[bookmarksnumbered,  plainpages]{hyperref}
\usepackage{multirow}
\usepackage{inputenc}

\newtheoremstyle{case}{}{}{}{}{}{:}{ }{}
\theoremstyle{case}

\newcommand{\be}{\begin{equation}}
\newcommand{\ee}{\end{equation}}
\newcommand{\ben}{\begin{eqnarray*}}
\newcommand{\een}{\end{eqnarray*}}
\newtheorem{examp}{\sc Example}
\newtheorem{remk}{\sc Remark}
\newtheorem{corol}{\sc Corollary}
\newtheorem{lemma}{\sc Lemma}
\newtheorem{theorem}{\sc Theorem}
\newtheorem{defn}{\sc Definition}
\newtheorem{prop}{\sc Proposition}
\newcommand{\bet}{\begin{theorem}}
    \newcommand{\eet}{\end{theorem}}
\newcommand{\bel}{\begin{lemma}}
    \newcommand{\eel}{\end{lemma}}
\newcommand{\bed}{\begin{defn}}
    \newcommand{\eed}{\end{defn}}
\newcommand{\brem}{\begin{remk}}
    \newcommand{\erem}{\end{remk}}
\newcommand{\bex}{\begin{examp}}
    \newcommand{\eex}{\end{examp}}
\newcommand{\bcl}{\begin{corol}}
    \newcommand{\ecl}{\end{corol}}
\newcommand{\bep}{\begin{prop}}
    \newcommand{\eep}{\end{prop}}

\newcommand{\NI}{\noindent}
\newcommand{\bea}{\begin{eqnarray}}
\newcommand{\eea}{\end{eqnarray}}

\newcommand{\vsp}{\vskip 1em}

\begin{document}
\title{\large {\bf{\sc Solution of nonlinear system of equations through homotopy path}}}
\author{ A. Dutta$^{a, 1}$ and  A. K. Das$^{b, 2}$\\
\emph{\small $^{a}$Department of Mathematics, Jadavpur University, Kolkata, 700 032, India}\\	
\emph{\small $^{b}$SQC \& OR Unit, Indian Statistical Institute, Kolkata, 700 108, India}\\
\emph{\small $^{1}$Email: aritradutta001@gmail.com}\\
\emph{\small $^{2}$Email: akdas@isical.ac.in} \\
 }
\date{}

\maketitle
\begin{abstract}
	The paper aims to show the equivalency between nonlinear complementarity problem and the system of nonlinear equations. We propose a homotopy method with vector parameter $\lambda$ in finding the solution of nonlinear complementarity problem through a system of nonlinear equations.  We propose a smooth and bounded homotopy path to obtain  solution of the system of nonlinear equations under some conditions. An oligopolistic market equilibrium problem is considered to show the effectiveness of the proposed homotopy continuation method.
	\vsp
\NI{\bf Keywords:} Nonlinear complementarity problem, system of nonlinear equations, homotopy function with vector parameter, bounded smooth curve, oligopolistic market equilibrium. \\

\NI{\bf AMS subject classifications:} 90C33, 90C30, 15A69.
\end{abstract}

\section{Introduction}
A class of problems arising in various fields of sciences, can be studied via the nonlinear system of equations using various techniques. In recent years, researchers are interested  to solve  nonlinear system of equations both analytically and numerically. Several iterative
 methods have been developed using different techniques such as Taylor’s series expansion,
quadrature formulas, homotopy method, interpolation, decomposition and its various modification.
For details, see \cite{sayevand2016systems}, \cite{noor2009some}, \cite{jafari}, \cite{ojika1983deflation}, \cite{fan2003modified} and \cite{yamamura1998interval}.

Eaves and Saigal \cite{eaves1972homotopies} formed an important class of globally convergent methods for solving systems of non-linear equations, which is known as homotopy method. Such methods have been used to constructively prove the existence of solutions to many economic and engineering problems. Let $Y,X$ be two topologocal spaces and  $p, q:Y \to X$ be continuous maps. A homotopy from $p$ to $q$ is a continuous function $H:Y \times [0,1] $ $ \to X$ satisfying $H(y,0) = p(y),$ $H(y,1) = q(y)  \ \forall \ y \in Y.$ If such a homotopy exists, then $p$ is homotopic to $q$ and it is denoted by $p \simeq q.$ Let $p,q:R\to R$ any two continuous, real functions, then $p \simeq q.$ Now we define a function $H:R\times [0,1] \to R $ by $H(y,u)=(1-u)f(y)+ug(y).$ Clearly $H$ is continuous and $H(y,0)=p(y),$ $H(y,1)=q(y).$ Thus $H$ is a homotopy between $p$ and $q.$ Let $Y,X$ be two topological spaces and  Map$(Y,X)$ be the set of all continuous maps from $Y$ to $X.$ Homotopy is an equivalence relation on Map$(Y,X).$
\vsp
 The fundamental idea of the homotopy continuation method is to solve a problem by tracing a certain continuous path that leads to a solution to the problem. Thus, defining a homotopy mapping that yields a finite continuation path plays an essential role in a homotopy continuation method. The homotopy method \cite{watson1989globally} is itself an important class of globally convergent methods. Many homotopy methods are proposed for constructive proof of the existence of solutions to systems of nonlinear equations, nonlinear optimization problems, Brouwer fixed point problems, nonlinear programming, game problem and complementarity problems \cite{watson1989modern}. 
 
 We are interested in solving the  nonlinear system of equations. The nonlinear complementarity problem, is identified as an important mathematical programming problem can be converted into nonlinear system of equations.The idea of nonlinear complementarity problem is based on the concept of linear complementarity problem. For recent study on this problem and applications see \cite{das2017finiteness}, \cite{articlee14}, \cite{bookk1}, \cite{articlee7} and references therein. For details of several matrix classes in complementarity theory, see \cite{articlee1}, \cite{articlee2}, \cite{articlee9}, \cite{articlee17}, \cite{article1}, \cite{mohan2001more}, \cite{article12}, \cite{article07}, \cite{dutta2022column} and references cited therein. The problem of computing the value vector and optimal stationary strategies for structured stochastic games for discounted and undiscounded zero-sum games and quadratic Multi-objective programming problem are formulated as linear complementary problems. For details see \cite{articlee18}, \cite{mondal2016discounted}, \cite{neogy2005linear} and \cite{neogy2008mixture}. The complementarity problems are considered with respect to principal pivot transforms and pivotal method to its solution point of view. For details see \cite{articlee8}, \cite{articlee10}, \cite{das1} and \cite{neogy2012generalized}.

In this paper, we consider the modified homotopy continuation method to find the solution of the nonlinear system of  equations with vector parameter. 
\vsp
The paper is organized as follows. Section 2 presents some basic notations and results. In section 3, we prove the necessary and sufficient condition, which is a system of nonlinear equations to find the solution of nonlinear complementarity problem. A new homotopy approach with vector parameter is proposed to find the solution of nonlinear system of equations. In the last  section, we  consider an oligopoly market equilibrium  problem \cite{murphy} is used to illustrate the effectiveness of the proposed homotopy approach with vector parameter.

\section{Preliminaries}

Consider a function $f: R^n \rightarrow R^n$ , and a vector 
$z\in R^n$ such that $f= \left[\begin{array}{c}
f_1\\
f_2\\
\vdots\\
f_n\\
\end{array}\right]$ and $z= \left[\begin{array}{c}
z_1\\
z_2\\
\vdots\\
z_n\\
\end{array}\right].$ 
The complementarity problem is to find a vector $z\in R^n$ such that
\begin{equation}\label{cp}
z^Tf(z)=0, \ \ \  f(z)\geq 0 , \ \ \   z\geq 0.
\end{equation}
When the function $f$ is a nonlinear function, then it is called nonlinear complementarity problem.  For details see \cite{karamardian1969}.

The basic idea of homotopy method is to construct a homotopy continuation path 
from the auxiliary mapping g to the object mapping f. 
Suppose the given problem is to find a root of the non-linear equation f(x) = 0
and suppose g(x) = 0 is an auxiliary equation with $g(x_0)=0$. Then the
homotopy function $H:R^{n+1} \to R^n$ can be
 defined as $H(x, \mu) = $
 $ (1-\mu)f(x) + \mu g(x),$
 $ 0 \leq \mu \leq 1$ is a parameter.
Then we consider the homotopy equation $H(x, \mu) = 0,$ where $(x_0,1)$ is a known solution of the homotopy equation. Our aim is to find the solution of the equation $f(x)=0$ from the known solution of $g(x) = 0$ by solving the homotopy equation $H(x, \mu) = 0$ varrying the values of $\mu$  from $1$ to $0$. 

Now we state some results.
\begin{lemma}
	(Generalizations of Sard's Theorem\cite{Chow}) \ Let $W \subset R^n$ be an open set and $f :R^n \to R^q$ be smooth. We say $z \in R^q$ is a regular value for $f$ if $\text{Range} Df(w) = R^q $ $\forall \ w \in f^{-1}(z),$ where $Df(w)$ denotes the $n \times q$ matrix of partial derivatives of $f(w).$
\end{lemma}
\begin{lemma}\label{par}
	(Parameterized Sard Theorem \cite{Wang}) \ Let $W \subset R^p, Z \subset R^q$ be open sets, and let $\phi:W\times Z \to R^m$ be a $C^\omega$ mapping, where $\omega >\text{max}\{0,q-m\}.$ If $0\in R^m$ is a regular value of $\phi,$ then for almost all $a \in W, 0$ is a regular value of $\phi _ w=\phi(w,.).$    
\end{lemma}

\begin{lemma}\label{inv}
	(The inverse image theorem \cite{Wang}) \ Let $\phi : W \subset R^p \to R^q$ be $C^\omega$ mapping, where $\omega >\text{max}\{0,p-q\}.$ Then $\phi^{-1}(0)$ consists of some $(p-q)$-dimensional $C^\omega$ manifolds. 
\end{lemma}
\begin{lemma}\label{cl}
	(Classification theorem of one-dimensional smooth manifold \cite{N}) \ One-dimensional smooth manifold is diffeomorphic to a unit circle or a unit interval.
\end{lemma}

Now in the next section we will solve the nonlinear complementarity problem  \ref{cp} using homotopy method with vector parameter.
\section{Main results}

 \begin{theorem}
    	Let $\phi:R \to R$ be any increasing function such that $\phi(0)=0$. Then $z$ solves the complementarity problem \ref{cp} if and only if 
    	\begin{equation}\label{ma}
    	\phi((f_i(z)-z_i)^2) - \phi(f_i(z) |f_i(z)|) - \phi(z_i|z_i|)=0
    	\end{equation}
       \end{theorem}
   \begin{proof}
   	Necessery. For each $i=1,2,\ldots ,n$, either $z_i=0 , f_i(z)\geq 0$  or $f_i(z)=0 , z_i \geq 0$.\\ If $z_i=0,f_i(z)\geq 0 $ then
   	$\phi((f_i(z)-z_i)^2) - \phi(f_i(z) |f_i(z)|) - \phi(z_i|z_i|) $\\
   	$ = \phi((f_i(z))^2) - \phi(f_i(z) |f_i(z)|)=\phi((f_i(z))^2) - \phi((f_i(z))^2)=0$.\\
   	If $f_i(z)=0 , z_i \geq 0$ then  $\phi((f_i(z)-z_i)^2) - \phi(f_i(z) |f_i(z)|) - \phi(z_i|z_i|) $\\ 
   	 	$ = \phi((z_i)^2) - \phi(z_i |z_i|)=\phi((z_i)^2) - \phi((z_i)^2)=0$.\\
   	 	So the solution of $(1.1)$ satisfies $(1.2)$.
   	 	\\
   	 	Sufficient. (a) To show that $f(z) \geq 0$ assume the contrary, i.e. $f_i(z) < 0$ for some $i=1,2,\dots,n$. Then
   	 	 \begin{center}
   	 	 $0\leq \phi((f_i(z)-z_i)^2)$
   	 	$= \phi(f_i(z)|f_i(z)|) + \phi(z_i|z_i|)$ $= \phi(-f_i(z)^2) + \phi(z_i|z_i|)$ $ < \phi(z_i|z_i|)$
   	 \end{center}
      This implies that $ \phi (z_i|z_i|) > 0$ $ \implies z_i|z_i| >0 $ $ \implies z_i >0$ and  $\phi((f_i(z)-z_i)^2)< \phi(z_i|z_i|)$ $\implies ((f_i(z)-z_i)^2) < z_i|z_i|$.\\
   	 	But for $z_i >0$,  $(z_i)^2 < (f_i(z)-z_i)^2$   [as $(f_i(z)-z_i)<0, z_i >0 $ then $ (f_i(z)-z_i)<0$, $|f_i(z)-z_i|>z_i$ so $(f_i(z)-z_i)^2 >z_i^2$ ].\\
   	 	 So, it contradicts that $f_i(z) < 0$. So $f_i(z) \geq 0$.\\
   	 	 (b)  To show that $z \geq 0$ interchange the roles of $ z_i$ and $f_i(z)$.\\
   	 	 (c)From (a) and (b) we have that  $z \geq 0$ and  $f(z) \geq 0$. To show $z^Tf(z)=0$ assume the contrary $z_i >0$ and $f_i(z)>0$ for some $ i =1,2,\ldots,n.$\\
   	 	 If $f_i(z)\geq z_i$, then $\phi( (f_i(z)-z_i)^2) < \phi( (f_i(z))^2) + \phi( (z_i)^2)= \phi(f_i(z) |f_i(z)|) + \phi(z_i|z_i|)$. This contradicts that 	$\phi((f_i(z)-z_i)^2) = \phi(f_i(z) |f_i(z)|) + \phi(z_i|z_i|) $.\\
   	 	 Similarly, for  $z_i(z)\geq f_i(z),$ interchanging the roles of $z_i$ and $f_i(z)$, we get a contradiction.
   \end{proof}

  Hence it is shown that the complementarity problem of finding a $z\in R^n$ satisfying $z^Tf(z)=0,\ \ \ f(z) \geq 0, \ \ \ z \geq 0$ where  $f: R^n \rightarrow R^n$ is equivalent to the problem of solving system of n nonlinear equations in n variables.
  	\begin{equation}{\label{ne}}
  \psi_i(z)=\phi((f_i(z)-z_i)^2) - \phi(f_i(z) |f_i(z)|) - \phi(z_i|z_i|)=0 \ \ \ \forall   i \in \{1,2,\cdots n\},
  \end{equation}
  where $\phi$ is an increasing function defined by $\phi:R \to R$ such that $\phi(x)=x^3.$
   Now $\frac{\partial \psi_i}{\partial z_j}$ $=\phi^{'}((f_i(z)-z_i)^2)2(f_i(z)-z_i)(\frac{\partial f_i}{\partial z_j}-\delta_{ij})$  $-\phi^{'}(f_i(z)|f_i(z)|)2f_i(z)sgn(f_i(z))\frac{\partial f_i}{\partial z_j}$  $-\phi^{'}(z_i|z_i|)2z_isgn(z_i) \delta_{ij}$
    where 
    \[   
    sgn(x) = 
    \begin{cases}
    \  1 &\quad\text{if } x  > 0\\
    \ -1 &\quad\text{if } x < 0\\
    \end{cases}
    \]
     \quad
    
    \begin{theorem}
      Let $z$ be the nondegenerate solution of nonlinear complementarity problem \ref{cp}, i.e. $z +f(z) >0$. Let $\cal{J}$ $(f (z))$, the Jacobian of $f$ at $z$, have
nonsingular principal minors and let $\phi: R \to R$ be a differentiable strictly increasing function such that $\phi'(y) > 0$ for all $y>0$ and $\phi(0)=0.$ Then $z$ solves \ref{ne} and
the Jacobian of $\psi$ at $z$, $\cal{J}$ $(\psi (z))$  is nonsingular.
  
    \end{theorem}
    \begin{proof}
    The Jacobian matrix of $\psi(z)$ is defined by $\cal{J}$ $(\psi (z))=$   $\left[\begin{array}{rrrr}
    		\frac{\partial \psi_1}{\partial z_1} &  \frac{\partial \psi_1}{\partial z_2}  & \cdots & \frac{\partial \psi_1}{\partial z_n} \\
    \frac{\partial \psi_2}{\partial z_1}	& \frac{\partial \psi_2}{\partial z_2} & \cdots & \frac{\partial \psi_2}{\partial z_n} \\
    \vdots & \vdots  & \vdots & \vdots\\
    \frac{\partial \psi_n}{\partial z_1} & \frac{\partial \psi_n}{\partial z_2} & \cdots & \frac{\partial \psi_n}{\partial z_n} 
    	\end{array}\right],$
    
 where  $\frac{\partial \psi_i}{\partial z_i}=\phi^{'}((f_i(z)-z_i)^2)2(f_i(z)-z_i)(\frac{\partial f_i}{\partial z_i} - 1)  -\phi^{'}(f_i(z)|f_i(z)|)2f_i(z)sgn(f_i(z))\frac{\partial f_i}{\partial z_i}  -\phi^{'}(z_i|z_i|)2z_i sgn(z_i)$ and 
  $\frac{\partial \psi_i}{\partial z_j}=\phi^{'}((f_i(z)-z_i)^2)2(f_i(z)-z_i)(\frac{\partial f_i}{\partial z_j})-$\\$\phi^{'}(f_i(z)|f_i(z)|)2f_i(z)sgn(f_i(z))\frac{\partial f_i}{\partial z_j}  $ for $i \neq j.$\\
 Assume that $f_i (z)=0$ for $i=1,2,\cdots, n_1, n_1\leq n$ and $f_i (z)>0$ for $i=n_1 +1, n_1 +2, \cdots n.$ Hence by nondegeneracy of $z$, $z_i>0$ for $i=1,2,\cdots n_1, n_1\leq n,$ $z_i=0$ for $i=n_1 +1, n_1 +2, \cdots n.$\\
 Now $\cal{J}$ $(\psi (z) )=$ $\left[\begin{array}{rrrr}
    		-\phi^{'}((z_1)^2)2z_1 & 0 & \cdots & 0\\ 
    		0 & -\phi^{'}((z_2)^2)2z_2 &  \cdots & 0 \\
    		\vdots & \vdots & \cdots & \vdots\\
    		0 & 0 &  \cdots & -\phi^{'}((z_{n_1})^2)2z_{n_1} \\
    	   		0 & 0 & \cdots & 0\\
    		\vdots & \vdots & \cdots & \vdots\\
    		0 & 0 & \cdots & 0\\
      	\end{array}\right]$ $\cal{J}$ $(f (z) )$ $+$ $\left[\begin{array}{rrrr}
      	0 & 0 & \cdots & 0\\
    		\vdots & \vdots & \cdots & \vdots\\
    		0 & 0 & \cdots & 0\\
    		-\phi^{'}(({f_{n_1 +1}}(z))^2)2{f_{n_1 +1}}(z) & 0 & \cdots & 0\\ 
    		0 & -\phi^{'}(({f_{n_1 +2}}(z))^2)2{f_{n_1 +2}}(z) &  \cdots & 0 \\
    		\vdots & \vdots & \cdots & \vdots\\
    		0 & 0 &  \cdots & -\phi^{'}(({f_n}(z))^2)2{f_n}(z) \\
    	   		0 & 0 & \cdots & 0\\
    		\vdots & \vdots & \cdots & \vdots\\
    		0 & 0 & \cdots & 0\\
      	\end{array}\right]$. Since $\cal{J}$ $(f(z) )$ has nonsingular principal minors and $\phi'(y)>0$ for $y>0,$ the Jacobian of $\psi(z),$ $\cal{J}$ $(\psi(z) )$  is nonsingular. 
\end{proof}

\subsection{Homotopy based solution with vector parameter $\tilde{\lambda} \in R^n$}
The basic idea of homotopy method with $\tilde{\lambda} \in R^n$ is to construct a multidimentional homotopy path to find the solution of the object fuction $\tilde{f}(\tilde{x})=0$ varying each component of $\tilde{\lambda}$ from $1$ to $0.$ We consider the homotopy function\\ $\mathcal H(\tilde{x},\tilde{x}^{(0)},\tilde{\lambda})=\tilde{f}(\tilde{x})-\tilde{\lambda}\tilde{f}(\tilde{x}^{(0)}),$ where each component of  $\tilde{f}(\tilde{x}^{(0)}), \ \  f_i(\tilde{x}^{(0)})\neq 0 \ \forall i \in \{1,2, \cdots, n\}$  and the product term $\tilde{\lambda}\tilde{f}(\tilde{x}^{(0)})$ is a componentwise product that is $\tilde{\lambda}\tilde{f}(\tilde{x}^{(0)})=\left[\begin{array}{c}
{\lambda_1}f_1(\tilde{x}^{(0)})\\
{\lambda_2}f_2(\tilde{x}^{(0)})\\
\vdots\\
{\lambda_n}f_n(\tilde{x}^{(0)})\\
\end{array}\right],$ where $\tilde{\lambda}=\left[\begin{array}{c}
{\lambda_1}\\
{\lambda_2}\\
\vdots\\
{\lambda_n}\\
\end{array}\right]$ and $\tilde{f}(\tilde{x}^{(0)})=\left[\begin{array}{c}
f_1(\tilde{x}^{(0)})\\
f_2(\tilde{x}^{(0)})\\
\vdots\\
f_n(\tilde{x}^{(0)})\\
\end{array}\right].$ \\ In this method our main aim is to vary each component $\lambda_i=\frac{f_i(\tilde{x})}{f_i(\tilde{x}^{(0)})}$ of $\tilde{\lambda}$ from $1$ to $0.$

Now we solve the nonlinear equation \ref{ma} using the homotopy function with vector parameter $\tilde{\lambda}=[\lambda_1,\lambda_2, \cdots, \lambda_n]^t \in R^n.$ Let $\psi_i(z)=\phi((f_i(z)-z_i)^2) - \phi(f_i(z) |f_i(z)|) - \phi(z_i|z_i|)=0.$ Consider the set $F_{\mathcal{H}}=\{\tilde{x}:\tilde{x}\in R^n\}$ and $\tilde{F}_{\mathcal{H}}=\{(\tilde{x}, \tilde{\lambda}):(\tilde{x}, \tilde{\lambda}) \in R^n \times (0,1]^n\}.$ The homotopy function can be given as, 
\begin{equation}{\label{hybrid}}
\mathcal H(\tilde{x},\tilde{x}^{(0)},\tilde{\lambda})=\psi(\tilde{x})- \tilde{\lambda}\psi(\tilde{x}^{(0)})=\left[\begin{array}{c}
\psi_1(\tilde{x})-{\lambda_1}\psi_1(\tilde{x}^{(0)})\\
\psi_2(\tilde{x})-{\lambda_2}\psi_2(\tilde{x}^{(0)})\\
\vdots\\
\psi_n(\tilde{x})-{\lambda_n}\psi_n(\tilde{x}^{(0)})
\end{array}\right]
\end{equation}
Where, $\tilde{x}^{(0)}$ is the initial known value of the vector $\tilde{x}\in R^n$ such that $ \psi_i(\tilde{x}^{(0)}) \neq 0 \  \forall i \ \in \{1,2, \cdots, n\}$  and $\tilde{\lambda} \in R^n.$ In this modification, $\tilde\lambda$ is taken as a vector.

\begin{theorem}\label{reg}
   If the Jacobian matrix $\frac{\partial{\psi}(\tilde{x}^{(0)})}{\partial\tilde{x}^{(0)}}$ at initial point $\tilde{x}^{(0)}$ is nonsingular, then for almost all initial points $\tilde{x}^{(0)}\in \tilde{F}_\mathcal{H},$ $0$ is a regular value of the homotopy function $\mathcal{H}:R^{6n} \times (0,1]^n \to R^{6n}$ and the zero point set $\mathcal{H}_{\tilde{x}^{(0)}}^{-1}(0)=\{(\tilde{x},\tilde{\lambda})\in \tilde{F}_{\mathcal{H}}:\mathcal{H}_{\tilde{x}^{(0)}}(\tilde{x},\tilde{\lambda})=0\}$ contains a smooth curve $\Gamma_{\tilde{x}^{(0)}}$ starting from $(\tilde{x}^{(0)},e).$  
\end{theorem}
\begin{proof}
		The Jacobian matrix of the above homotopy function\ref{hybrid} $\mathcal{H}(\tilde{x},\tilde{x}^{(0)},\tilde{\lambda})$ is denoted by $D\mathcal{H}(\tilde{x},\tilde{x}^{(0)},\tilde{\lambda}))$ and we have\\ $D\mathcal{H}(\tilde{x},\tilde{x}^{(0)},\tilde{\lambda}))=$ $\left[\begin{array}{ccc} 
	\frac{\partial{\mathcal{H}(\tilde{x},\tilde{x}^{(0)},\tilde{\lambda})}}{\partial{\tilde{x}}} & 	\frac{\partial{\mathcal{H}(\tilde{x},\tilde{x}^{(0)},\tilde{\lambda})}}{\partial{\tilde{x}^{(0)}}} & \frac{\partial{\mathcal{H}(\tilde{x},\tilde{x}^{(0)},\tilde{\lambda})}}{\partial{\tilde{\lambda}}}\\ 
	\end{array}\right].$\\ For all $\tilde{x}^{(0)} \in \tilde{F}_H$ such that $\psi(\tilde{x}^{(0)}) \neq 0$ and $\tilde{\lambda} \in (0,1]^n,$ we have\\ $\frac{\partial{\mathcal{H}(\tilde{x},\tilde{x}^{(0)},\tilde{\lambda})}}{\partial{\tilde{x}^{(0)}}}=$
	$ \left[\begin{array}{cccc}
	-{\lambda_1}\frac{\partial{\psi_1}(\tilde{x}^{(0)})}{\partial \tilde{x}_1^{(0)}} \ -{\lambda_1}\frac{\partial{\psi_1}(\tilde{x}^{(0)})}{\partial \tilde{x}_2^{(0)}} \ \cdots \ -{\lambda_1}\frac{\partial{\psi_1}(\tilde{x}^{(0)})}{\partial \tilde{x}_n^{(0)}} \\
	-{\lambda_2}\frac{\partial{\psi_2}(\tilde{x}^{(0)})}{\partial \tilde{x}_1^{(0)}} \ -{\lambda_2}\frac{\partial{\psi_2}(\tilde{x}^{(0)})}{\partial \tilde{x}_2^{(0)}} \ \cdots \ -{\lambda_2}\frac{\partial{\psi_2}(\tilde{x}^{(0)})}{\partial \tilde{x}_n^{(0)}}\\
	\vdots \ \vdots \ \vdots \ \vdots \\
	-{\lambda_n}\frac{\partial{\psi_n}(\tilde{x}^{(0)})}{\partial \tilde{x}_1^{(0)}} \ -{\lambda_n}\frac{\partial{\psi_n}(\tilde{x}^{(0)})}{\partial \tilde{x}_2^{(0)}} \ \cdots \ -{\lambda_n}\frac{\partial{\psi_n}(\tilde{x}^{(0)})}{\partial \tilde{x}_n^{(0)}}\end{array}\right].$\\ 
	Now $\det\frac{\partial{\mathcal{H}(\tilde{x},\tilde{x}^{(0)},\tilde{\lambda})}}{\partial{\tilde{x}^{(0)}}}= $
	$(-1)^n\det J_{\mathcal{H}}{\tilde{x}^{(0)}}\prod_{i=1}^{i=n}\lambda_i \neq 0$ for $\lambda \in (0; 1]^n,$ where $J_{\mathcal{H}}{\tilde{x}^{(0)}}=\frac{\partial{\psi}(\tilde{x}^{(0)})}{\partial\tilde{x}^{(0)}}.$  Thus $D\mathcal{H}(\tilde{x},\tilde{x}^{(0)},\tilde{\lambda}))$ is of full  row rank. Therefore, $0$ is a regular value of $\mathcal{H}(\tilde{x},\tilde{x}^{(0)},\tilde{\lambda})$ and  $\mathcal{H}_{\tilde{x}^{(0)}}^{-1}(0)$ consists of some smooth curves such that $\mathcal{H}_{\tilde{x}^{(0)}}(\tilde{x}^{(0)},1)=0.$ Hence there must be a smooth curve  $\Gamma_{\tilde{x}^{(0)}}$ starting from  $(\tilde{x}^{(0)},1).$
  \end{proof}

\begin{theorem}
Assume that the function $f(\tilde{x})$ is either increasing or bounded function.  The smooth curve $\Gamma_{\tilde{x}^{(0)}}$  is a bounded curve. 
\end{theorem}
\begin{proof}
	 From the homotopy function \ref{hybrid} we get $\psi_i(\tilde{x})={\lambda_i}\psi_i(\tilde{x}^{(0)}) \  \forall i \in \{1,2,\cdots,n \},$where $\psi_i(\tilde{x})=	\phi((f_i(\tilde{x})-\tilde{x}_i)^2) - \phi(f_i(\tilde{x}) |f_i(\tilde{x})|) - \phi(\tilde{x}_i|\tilde{x}_i|).$ It is clear that $\lambda_i \in [0,1],$ $\psi_i(\tilde{x})$ is finite. So $\|\psi(\tilde{x})\|<\infty.$ \\This implies that $\psi_i(\tilde{x})=	\phi((f_i(\tilde{x})-\tilde{x}_i)^2) - \phi(f_i(\tilde{x}) |f_i(\tilde{x})|) - \phi(\tilde{x}_i|\tilde{x}_i|)={\lambda_i}\psi_i(\tilde{x}^{(0)}) \  \forall i \in \{1,2,\cdots,n \}.$ Let ${\lambda_i}\psi_i(\tilde{x}^{(0)}) =k_i.$ For $\lambda_i \in (0,1], k_i>0.$\\ Hence $\phi((f_i(\tilde{x})-\tilde{x}_i)^2) = \phi(f_i(\tilde{x}) |f_i(\tilde{x})|) + \phi(\tilde{x}_i|\tilde{x}_i|)+k_i.$ \\ Again $\forall i \in \{1,2, \cdots, n\}$ we have  $\psi_i(\tilde{x})=((f_i(\tilde{x})-\tilde{x}_i)^6)-(f_i(\tilde{x}) |f_i(\tilde{x})|)^3-(\tilde{x}_i|\tilde{x}_i|)^3=(f_i(\tilde{x}))^6-6(f_i(\tilde{x}))^5\tilde{x}_i+15(f_i(\tilde{x}))^4(\tilde{x}_i)^2-20(f_i(\tilde{x}))^3(\tilde{x}_i)^3+15(f_i(\tilde{x}))^2(\tilde{x}_i)^4-6(f_i(\tilde{x}))(\tilde{x}_i)^5+(\tilde{x}_i)^6-(f_i(\tilde{x}))^5|f_i(\tilde{x})|-(\tilde{x}_i)^5|(\tilde{x}_i)|.$ Consider the homotopy path $\Gamma_{\tilde{x}^{(0)}}$  is unbounded. Then there exists a sequence of points $(\tilde{x}^{(k)}, \tilde{\lambda}^{(k)}) \subset\Gamma_{\tilde{x}^{(0)}}$ such that $\|\tilde{x}^{(k)}\| \to \infty.$  Then there are two possibilities. \\
	 \vsp
	\NI	\textbf{Case 1:} Let $\|\tilde{x}^{(k)}\| \to \infty.$ Then $\exists \ i \in \{1,2,\cdots,n\}$ such that $\tilde{x}^{(k)}_i \to -\infty$ as $k \to \infty.$ Let set $I_{1\tilde{x}}=\{i \in \{1,2,\cdots,n\}:\lim\limits_{k \to \infty}\tilde{x}^{(k)}_i \to -\infty\}.$ \\
		\smallskip
		Now consider  $\|f(\tilde{x}^{(k)})\|< \infty.$  Then from above we have $\forall i \in I_{1\tilde{x}},$ as $k \to \infty,$ $\frac{\psi_i(\tilde{x}^{(k)})}{(\tilde{x}^{(k)} _i)^6} \to 2,$ which contradicts that $\psi(x)$ is bounded.\\
		\smallskip
		Again consider that the nonlinear function $f(\tilde{x})$ is unbounded. It is noted that $f(\tilde{x})$ is an increasing function. \\Let $L_{1f(\tilde{x})}=\{l \in \{1,2,\cdots,n\}:\lim\limits_{k \to \infty}f_l (\tilde{x}^{(k)}) \to -\infty\}$ and consider $\lim\limits_{k\to \infty}\frac{f_i(\tilde{x}^{(k)})}{(\tilde{x}_i ^{(k)})} =p \ \forall \ i \in I_{1\tilde{x}} \cap L_{1f(\tilde{x})} $. 
		Therefore $ \forall \ i \in I_{1\tilde{x}} \cap L_{1f(\tilde{x})},$ as $k \to \infty,$ $ \frac{\psi_i(\tilde{x}^{(k)})}{(\tilde{x}^{(k)} _i)^6} \to 2p^6 -6p^5+15p^4-20p^3+15p^2-6p+2.$ From the boundedness of $\psi(x^{(k)})$, it is clear that $2p^6 -6p^5+15p^4-20p^3+15p^2-6p+2=0,$  has no real solution, which is a contradiction.\\
		Now for all $l \in L_{1f(\tilde{x})}, l \notin I_{1\tilde{x}},$ $\lim\limits_{k \to \infty}\frac{\psi_l(\tilde{x}^{(k)})}{f_l(\tilde{x}^{(k)})^6} \to 2,$ contradicts the boundedness of the function $\psi(\tilde{x}).$ \\

		\NI	\textbf{Case 2:}  Let $\|\tilde{x}^{(k)}\| \to \infty.$ Then $\exists \ j \in \{1,2,\cdots,n\}$ such that $\tilde{x}^{(k)}_j \to \infty$ as $k \to \infty.$ Let set $I_{2\tilde{x}}=\{j \in \{1,2,\cdots,n\}:\lim\limits_{k \to \infty}\tilde{x}^{(k)}_j \to \infty\}$.\\
			Let $\tilde{x}_j > 0, {f_j}(\tilde{x}) \geq 0.$ Consider that $\lambda_j \in(0,1].$Then $\phi((f_j(\tilde{x})-\tilde{x}_j)^2) = \phi(f_j(\tilde{x}) |f_j(\tilde{x})|) + \phi(\tilde{x}_j|\tilde{x}_j|)+k_j=\phi((f_j(\tilde{x})^2) + \phi((\tilde{x}_j)^2)+k_j,$ where $k_j > 0.$ Now $0 \leq \phi((f_j(\tilde{x})-\tilde{x}_j)^2)\leq\phi((f_j(\tilde{x})^2)+\phi((\tilde{x}_j)^2).$ This contradicts that $\phi((f_j(\tilde{x})-\tilde{x}_j)^2)=\phi((f_j(\tilde{x})^2) + \phi((\tilde{x}_j)^2)+k_j.$ Therefore if $\psi_j(\tilde{x})={\lambda_j}\psi_j(\tilde{x}^{(0)}), $ $\lambda_j \in(0,1]$ has the solution $\tilde{x}_j>0,$ then $f_j(\tilde{x}) < 0.$\\
		Now consider that $\|f(\tilde{x}^{(k)})\|<\infty.$\\   Then $\forall j \in I_{2\tilde{x}},$ as $k \to \infty,$ $\frac{\psi_j(\tilde{x}^{(k)})}{(\tilde{x}^{(k)} _j)^5} \to -6(f_j(\tilde{x}^{(k)})) \nrightarrow 0, $ which is a contradiction.\\
Therefore considering both the cases we conclude that the homotopy path $\Gamma_{\tilde{x}^{(0)}}$ is bounded for $\lambda_i \in (0,1] \ \forall i$.
		\end{proof}
 Note that here we use the vector $\tilde{\lambda},$ such that $\tilde{\lambda}:[0,\infty]^n \to (0,1]^n$ with $ \lambda_i=\exp(-t_i).$

\begin{remk}
	Now we trace the  homotopy path $\Gamma_{\tilde{x}^{(0)}} \subset \mathcal{H}_{\tilde{x}^{(0)}}^{-1}(0) \subset \tilde{F}_{\mathcal{H}}$ from the initial point $(\tilde{x}^{(0)},e)$ until $\tilde{\lambda} \to \tilde{0}$ and find the solution of the system of nonlinear equation \ref{ma} under some assumptions. If the  homotopy path is bounded and $\tilde{\lambda}$ goes to $0 $ starting from $e,$ the component $\tilde{x}$ give the solution of \ref{ma}.  Let $s$ denote the arc length of $\Gamma_{\tilde{x}^{(0)}}$, we can parameterize the  homotopy path $\Gamma_{\tilde{x}^{(0)}}$ with respect to $s$ in the following form\\
	\begin{equation}\label{sss}
	\mathcal{H}_{\tilde{x}^{(0)}} (\tilde{x}(s),\tilde{\lambda} (s))=0, \ 
	\tilde{x}(0)=\tilde{x}^{(0)}, \   \tilde{\lambda}(0)=e.
	\end{equation} 
	Now differentiating \ref{sss} with respect to $s$, we obtain the following system of ordinary differential equations with given initial values
	\begin{equation}\label{ivv}
	\mathcal{H}'_{\tilde{x}^{(0)}} (\tilde{x}(s),\tilde{\lambda} (s))\left[\begin{array}{c} 
	\frac{d\tilde{x}}{ds}\\
	\frac{d\tilde{\lambda}}{ds}\\
	\end{array}\right]=0, \
	\|( \frac{d\tilde{x}}{ds},\frac{d\tilde{\lambda}}{ds})\|_2=1, \ 
	\tilde{x}(0)=\tilde{x}^{(0)}, \   \tilde{\lambda}(0)=e. \ 
	\end{equation} 
	and the $\tilde{x}$-component of $(\tilde{x}(\bar{s}),\tilde{\lambda} (\bar{s}))$ gives the solution of the complementarity problem for $\tilde{\lambda}({\bar{s}})=0.$
\end{remk}

\begin{theorem}
	The homotopy path  $\Gamma_{\tilde{x}^{(0)}}$ is determined by the initial value problem \ref{ivv}.
\end{theorem}
\begin{proof}
	Let $s$ denote the arc length of the homotopy path $\Gamma_{\tilde{x}^{(0)}}.$ Now differentiating the homotopy equation \ref{sss}, we get $\frac{\partial\mathcal{H}}{\partial \tilde{x}}\frac{d\tilde{x}}{ds}+\frac{\partial\mathcal{H}}{\partial \tilde{\lambda}}\frac{d\tilde{\lambda}}{ds}=0.$ Let $\nu=\left[\begin{array}{c} 
	\tilde{x}\\
	\tilde{\lambda}\\
	\end{array}\right].$ As $s$ is the arc length of $\Gamma_{\tilde{x}^{(0)}},$ then $	\|\nu'\|=\|(\frac{d\tilde{x}}{ds},\frac{d\tilde{\lambda}}{ds})\|_2=1.$ Then we get the following system of ordinary differential equation.
	\begin{equation}\label{ivvvv}
	\left[\begin{array}{cc} 
\frac{\partial\mathcal{H}}{\partial \tilde{x}} & 
	\frac{\partial\mathcal{H}}{\partial \tilde{\lambda}}\\
	\end{array}\right]\mu=0, \mu^t\mu=1, \frac{d\nu}{ds}=\mu, \nu(0)=\left[\begin{array}{c}
	\tilde{x}^{(0)}\\
	e\\
	\end{array}\right]
	\end{equation}
	Hence from the system \ref{ivvvv} the first two equations are solvable on $\mu.$ Solving the following cauchy problem, the homotopy curve $\Gamma_{\tilde{x}^{(0)}}$ can be derived.
	\begin{equation}
	\frac{d\nu}{ds}=\mu, \nu(0)=\left[\begin{array}{c}
	\tilde{x}^{(0)}\\
	e\\
	\end{array}\right].
	\end{equation}  
 \end{proof}
\newpage
\subsection{Algorithm}\label{homhybrid}
\begin{algorithm}
\caption{ \ \textbf{\textit{ Homotopy Continuation Method}}}\label{allgo}
\NI \textbf{Step 0:} Set $i=i_c=0.$ [$i$ is the Number of Iteration(s) and $i_c$ is the Number of shifting of the Initial Point(s).] Give an initial point $(\tilde{x}^{(0)},\tilde{\lambda}_0) \in \tilde{F}_{\mathcal{H}} \times \{1\}^n.$ Set $\eta_1=10^{-12},c_0=50.$ 

$\kappa_1=\sqrt{2}, \kappa_2=9000, $ where the step-length is determined by $\kappa_1^k, k \in Z$ and the limit of the maximum step-length is maintained by $\kappa_1^k \leq \kappa_2.$

Set $\tilde{u}_1=e,$ $\tilde{p}_1=\tilde{\lambda}^{-1}(e)$ and $\tilde{p}_0=\tilde{\lambda}^{-1}(\tilde{0}),$ where $e=\left[\begin{array}{c} 
1\\
1\\
\vdots\\
1\\
\end{array}\right],$ vector of all $1'$s and $\tilde{0}=\left[\begin{array}{c} 
0\\
0\\
\vdots\\
0\\
\end{array}\right],$ vector of all $0'$s.

Set $\epsilon_1=10^{-16}e, \epsilon_2=10^{-10}e.$ These are real numbers, used as thresholds for $\tilde{\lambda}.$  If $\tilde{\lambda}$ achieves a value $0\leq \tilde{\lambda} \leq \epsilon_1,$ then the algorithm stops with an acceptable solution. If $\tilde{\lambda}$ achieves a value, such that, $\epsilon_1< \tilde{\lambda} \leq \epsilon_2,$ the algorithm stops, declaring that point as probable solution.

\NI \textbf{Step 1:}  Set $\left[\begin{array}{c} 
\tilde{x}\\
\tilde{t}\\
\end{array}\right]=\left[\begin{array}{c} 
\tilde{x}^{(0)}\\
\tilde{p_1}\\
\end{array}\right].$ Now calculate the constant $d^{(0)}=\det  (\frac{\partial{\mathcal{H}}}{\partial{\tilde{x}}}(\tilde{x}^{(0)},\tilde{\lambda}_0)) $ and $s^{(0)}=\frac{d{\tilde{\lambda}}}{dt}(t=p_1).$ If $|d^{(0)}|\leq \epsilon$ or $|s^{(0)}|\leq \epsilon,$  then stop else go to  step $2$ [$\epsilon \to 0,$ a threshold].

\NI\textbf{Step 2:} Set $c_1=c_2=0.$  Calculate the constant $d=\det (\frac{\partial{\mathcal{H}}}{\partial{\tilde{x}}}(\tilde{x},\tilde{\lambda})), N=\frac{d{\tilde{\lambda}(t)}}{dt},$ and $s=Ne.$ If $|d|\leq \epsilon,$ then stop else go to  step $3$.
\end{algorithm}

\begin{algorithm}
\NI\textbf{Step 3:} Determine the unit predictor direction $\tau^{(n)}$ by the following method:
If sign$(d)=-\text{sign}(d_0),$ then set $\tilde{t}_d=N^{-1}(e-\tilde{\lambda}(t)),$ else set $\tilde{t}_d=-N^{-1}\tilde{\lambda}(t).$ Calculate $\tilde{w}_d=-(\frac{\partial{\mathcal{H}}}{\partial{\tilde{x}}}(\tilde{x},\tilde{\lambda}))^{-1}(\frac{\partial{\mathcal{H}}}{\partial{\tilde{\lambda}}}(\tilde{x},\tilde{\lambda}))\tilde{t}_d,$ $\tau^{(n)}=$ $\left[\begin{array}{c} 
\tilde{x}_n\\
\tilde{t}_n\\
\end{array}\right]=\frac{1}{\|\tilde{x}_d,\tilde{t}_d\|}\left[\begin{array}{c} 
\tilde{x}_d\\
\tilde{t}_d\\
\end{array}\right],$ $\tilde{\tau}=\dfrac{\|\tilde{t}_d\|}{\|\tilde{x}_d,\tilde{t}_d\|},$ where $\|\tilde{x}_d,\tilde{t}_d\|=\sqrt{\tilde{x}_d^2+\tilde{t}_d^2}.$
If $\tilde{\tau} \leq \eta_1,$ then set $c_1=c_1+1$ else reset $c_1=0.$
If $c_1<c_0,$ then go to step $4$ else, if $\tilde{u}_1\leq \epsilon_2$ then, stop with a probable solution else, stop due to non-convergence.\\

\NI \textbf{Step 4:} Choosing step length: Set $ k=0; \gamma=[\nabla\mu(\tilde{x})]^t\tilde{x}_n,$ where $\mu: R^n \to R,$   is used to increase step length in the descent direction(s). Set
this Function, such that, $\mu(\bar{\tilde{x}})\leq \mu(\tilde{x}),$ and $\tilde{x}, \bar{\tilde{x}} \in \mathcal{F}_{\mathcal{H}},$ where $\bar{\tilde{x}}$ is the solution of the problem. $\mu(\tilde{x})$ is taken as $[\mathcal{H}_0(\tilde{x})]^t[\mathcal{H}_0(\tilde{x})],$ where $\mathcal{H}_0(\tilde{x})=\left[\begin{array}{c} 
\psi_1(\tilde{x})\\
\psi_2(\tilde{x})\\
\vdots\\
\psi_n(\tilde{x})
\end{array}\right].$ 

If $\gamma \geq 0, $ $\tilde{x}+\kappa_1^{k+1}\tilde{x}_n \in \mathcal{F}_{\mathcal{H}}, \min(\tilde{p}_0,\tilde{p}_1)<\tilde{t}+\kappa_1^{k+1}\tilde{t}_n<\max(\tilde{p}_0, \tilde{p}_1), $ then set $k=k+1$ and go to step $5.$

else if $\gamma< 0, \mu(\tilde{x}+\kappa_1^{k+1}\tilde{x}_n)<\mu(\tilde{x}+\kappa_1^{k}\tilde{x}_n), \tilde{x}+\kappa_1^{k+1}\tilde{x}_n \in \mathcal{F}_{\mathcal{H}},\min(\tilde{p}_0,\tilde{p}_1)<\tilde{t}+\kappa_1^{k+1}\tilde{t}_n<\max(\tilde{p}_0, \tilde{p}_1), $ then set $k=k+1,$ and go to step $5.$

else reset $c_2=0,$ and jump to step $6.$

\NI \textbf{Step 5:} If $\kappa_1^k>\kappa_2,$ then set $k=k-1,$ $c_2=c_2+1$ and go to step $6$, else go to step $4$.

\NI\textbf{Step 6:} If $c_2<c_0,$ then go to step $7$, else if $\tilde{u}_1\leq \epsilon_2,$ then stop with probable solution else stop.

\NI\textbf{Step 7:}  Compute the predictor and corrector point:
$\left[\begin{array}{c} 
\tilde{x}_p\\
\tilde{t}_p\\
\end{array}\right]= \left[\begin{array}{c} 
\tilde{x}\\
\tilde{t}\\
\end{array}\right]+\kappa_1^k\left[\begin{array}{c} 
\tilde{x}_n\\
\tilde{t}_n\\
\end{array}\right],$  $\left[\begin{array}{c} 
\tilde{x}_c\\
\tilde{t}_c\\
\end{array}\right]= \left[\begin{array}{c} 
\tilde{x}_p\\
\tilde{t}_p\\
\end{array}\right]-[J_{\mathcal{H}}(\tilde{x}_p,\tilde{t}_p)^{+}H(\tilde{x}_p,\tilde{t}_p)],$ where $[J_{\mathcal{H}}(\tilde{x}_p,\tilde{t}_p)^{+}]$ is the Moore-Penrose Inverse. If $\det(\text{diag}(\mathcal{H}(\tilde{x}_c,\tilde{p}_0)-\mathcal{H}(\tilde{x}_c,\tilde{p}_1))) \leq \epsilon,$ where $\epsilon \to 0,\text{a threshold},$ then stop else $\tilde{u}_s=(\text{diag}(\mathcal{H}(\tilde{x}_c,\tilde{p}_0)-\mathcal{H}(\tilde{x}_c,\tilde{p}_1))^{-1}\mathcal{H}(\tilde{x}_c,\tilde{p}_0).$ Now if $(\tilde{x}_c,\tilde{u}_s) \in \tilde{\mathcal{F}}_{\mathcal{H}},$ jump to step $10,$ else set $k=k-1$ and go to step $8.$ 
\end{algorithm}
\begin{algorithm}

\NI\textbf{Step 8:} Calculate $a = \text{min}(\kappa_1^k,\|\tilde{x}-\tilde{x}_c\|).$ If $a\leq \eta_2,$ then go to step $9$ else jump back to step $5.$

\NI\textbf{Step 9:} If $\tilde{u}_1 \leq \epsilon_2,$ then stop with a Probable Solution else, set $ i_c = i_c + 1$ and
jump back to Step $1$, after changing the Initial Point as, $\tilde{x}^{(0)} = \tilde{x}_c.$

\NI\textbf{Step 10:} Set $\tilde{t}_s=\tilde{\lambda}^{-1}(\tilde{u}_s),$ $\left[\begin{array}{c} 
\tilde{x} \\
\tilde{t} \\
\end{array}\right]=$
$\left[\begin{array}{c} 
\tilde{x}_c\\
\tilde{t}_s\\
\end{array}\right], \tilde{u}_1=\frac{\|\tilde{u}_s\|}{\sqrt{n}}$ 
If $\tilde{u}_1\leq \epsilon_1,$ then stop with acceptable  homotopy solution else set $i=i+1$ and go to step $2.$\\
\end{algorithm}
\bigskip
\newpage
\begin{theorem}
	If the  homotopy curve $\Gamma_{\tilde{x}}^{(0)}$ is smooth, then the positive predictor direction $\tilde{\tau}^{(0)}$ at the initial point $\tilde{x}^{(0)}$ satisfies sign($\det \left[\begin{array}{c}
	\frac{\partial \mathcal{H}}{\partial \tilde{x} \partial \tilde{\lambda}}(\tilde{x}^{(0)},1)\\
	e\tilde{\tau}^{(0)^t}\\
	\end{array}\right]$)$=(-1)^n\text{sign}\det J_{\mathcal{H}}{\tilde{x}^{(0)}},$ where $J_{\mathcal{H}}{\tilde{x}^{(0)}}=\frac{\partial{\psi}(\tilde{x}^{(0)})}{\partial\tilde{x}^{(0)}}.$ 
\end{theorem}
\begin{proof}
	From equation \ref{hybrid} we have\\
	$ \mathcal{H}(\tilde{x},\tilde{x}^{(0)},\tilde{\lambda})=$
$\left[\begin{array}{c}
\psi_1(\tilde{x})-{\lambda_1}\psi_1(\tilde{x}^{(0)})\\
\psi_2(\tilde{x})-{\lambda_2}\psi_2(\tilde{x}^{(0)})\\
\vdots\\
\psi_n(\tilde{x})-{\lambda_n}\psi_n(\tilde{x}^{(0)})
\end{array}\right]=0.$\\
	Now
	$\frac{\partial \mathcal{H}}{\partial \tilde{x} \partial \tilde{\lambda}}(\tilde{x}=\tilde{x}^{(0)},\tilde{\lambda}=e)=$$\begin{bmatrix} 
	J_{\mathcal{H}}\tilde{x}^{(0)} & \Psi(\tilde{x}^{(0)}) \\
    \end{bmatrix},$\\ where $\Psi(\tilde{x}^{(0)})=\begin{bmatrix}
    -\psi_1(\tilde{x}^{(0)}) & 0 & 0 & \cdots & 0\\
    0 & -\psi_2(\tilde{x}^{(0)}) & 0 & \cdots & 0\\
    \vdots & \vdots & \vdots & \vdots & \vdots\\
    0 & 0 & 0 & \cdots & -\psi_n(\tilde{x}^{(0)})\\
     \end{bmatrix}.$
	Let positive predictor direction be $\tau^{(0)}=\left[\begin{array}{c}
	\tilde{t} \\ -e
	\end{array}\right]=\left[\begin{array}{c}
	(\tilde{R}^{(0)}_1)^{(-1)}\tilde{R}_2^{(0)}e \\ -e
	\end{array}\right],$ where $\tilde{R}^{(0)}_1=J_{\mathcal{H}}\tilde{x}^{(0)}$  and $\tilde{R}^{(0)}_2=\Psi(\tilde{x}^{(0)}).$ Here $\text{det}(\tilde{R}^{(0)}_1)\neq 0.$  Therefore, 
	$\det\left[\begin{array}{c}
	\frac{\partial \mathcal{H}}{\partial \tilde{x}^{(0)} \partial \tilde{\lambda}}\\
	e\tilde{\tau} ^{(0)^t}\\
	\end{array}\right]$$=\det\left[\begin{array}{cc}
	R^{(0)}_1 & R^{(0)}_2\\
	ee^t(R^{(0)}_2)^t(R^{(0)}_1)^{(-t)} & -ee^t\\	
	\end{array}\right]$ \\ $= \det\left[\begin{array}{cc}
	R^{(0)}_1 & R^{(0)}_2\\
	0 & -ee^t-ee^t(R^{(0)}_2)^t(R^{(0)}_1)^{(-t)}(R^{(0)}_1)^{(-1)}R_2^{(0)} \\	
	\end{array}\right] \\ =\det(R^{(0)}_1)\det(-ee^t-ee^t(R^{(0)}_2)^t(R^{(0)}_1)^{(-t)}(R^{(0)}_1)^{(-1)}R_2^{(0)}) \\ =(-1)^n\det(R^{(0)}_1)\det(ee^t(I+\mathcal{A}^t\mathcal{A})),  $ where $\mathcal{A}=(R^{(0)}_1)^{(-1)}R_2^{(0)}.$ So\\ sign($\det \left[\begin{array}{c}
	\frac{\partial \mathcal{H}}{\partial \tilde{x} \partial \tilde{\lambda}}(\tilde{x}^{(0)},1)\\
	e\tilde{\tau}^{(0)^t}\\
	\end{array}\right]$)$=(-1)^n\text{sign}\det J_{\mathcal{H}}{\tilde{x}^{(0)}}.$ 
\end{proof}

\section{Numerical Example}\label{oliex}
To illustrate the use of the above three algorithms presented in the previous section, a numerical example\cite{murphy} is presented here. \\
Following Murphy et al. \cite{murphy} the oligopolistic market equilibrium problem follows:

Let there be $n$ firms, which supply a homogeneous product in a noncooperative fashion. Let $P(\tilde{Q}), \tilde{Q} \geq 0$ denote the inverse demand, where $\tilde{Q}= \sum_{i=1}^{n} Q_i,$ $Q_i \geq 0$ denote the $i$th firm's supply. Let $c_i(Q_i)$ be the total cost of supplying $Q_i$ units. Now the Nash equilibrium solution is a set of nonnegative output levels ${Q_1}^*,{Q_2}^*,\cdots,{Q_n}^*,$ such that ${Q_i}^*$ is an optimal solution to the following problem $\forall i\in \{1,2\cdots, n\}:$
\begin{eqnarray}\label{obj}
\text{maximize} \ {Q_iP(Q_i+{\tilde{Q}_i}^*)-c_i(Q_i)}
\end{eqnarray}
where ${\tilde{Q}_i}^*= \sum_{j \neq i}{Q_j}^*.$
Murphy et al. show that if $c_i(Q_i)$ is convex and continuously differentiable $\forall i \in\{1,2, \cdots,n\} $ and the inverse demand function $P(\tilde{Q})$ is strictly decreasing and continuously differentiable and the industry revenue curve $\tilde{Q}P(\tilde{Q})$ is concave, then \\
$({Q_1}^*,{Q_2}^*,\cdots,{Q_n}^*)$ is a Nash equilibrium solution if and only if
\begin{eqnarray}
[P(\tilde{Q}^*)+{Q_i}^*P'(\tilde{Q}^*)-{c_i}'({Q_i}^*)]{Q_i}^*=0 \\
{c_i}'({Q_i}^*)-P(\tilde{Q}^*)-{Q_i}^*P'(\tilde{Q}^*) \geq 0 \label{22}\\
{Q_i}^* \geq 0 \ \    \forall i\in \{1,2,\cdots,n\}
\end{eqnarray}

\noindent where $\tilde{Q}^*=\sum_{i=1}^{n}{Q_i}^*,$ which is a nonlinear complementarity problem with $f_i(z)={c_i}'({Q_i}^*)-P(\tilde{Q}^*)-{Q_i}^*P'(\tilde{Q}^*),  $ and $z_i= {Q_i}^*.$ \\
Note that here the functions $c_i(Q_i)$ and $-\tilde{Q}P(\tilde{Q})$ are convex. So the 1st order derivative of these two functions are increasing function. Hence the function $f_i(z)={c_i}'({Q_i}^*)-P(\tilde{Q}^*)-{Q_i}^*P'(\tilde{Q}^*)  $ is an increasing function.

Now consider an oligopoly with five firms, each with a total cost function of the form:
\begin{equation}\label{olix}
c_i(Q_i)=n_iQ_i+\frac{\beta_i}{\beta_i+1}{L_i}^{\frac{1}{\beta_i}}{Q_i}^{\frac{\beta_i+1}{\beta_i}}
\end{equation}
The demand curve is given by:
\begin{equation}
\tilde{Q}=5000P^{-1.1}, \ \   P(\tilde{Q})=5000^{1/1.1}\tilde{Q}^{-1/1.1}.
\end{equation}

\newpage
The parameters of the equation \ref{olix} for the five firms are given below: 
\begin{table}[ht]
	\caption{Value of parameters for five firms} 
	\centering
	\begin{tabular}{c c c c} 
		\hline\hline
		firm $i$ & $n_i$ & $L_i$ & $\beta_i$ \\ [0.5ex] 
		\hline
		1 & 10 & 5 & 1.2 \\ 
		2 & 8 & 5 & 1.1 \\
		3 & 6 & 5 & 1 \\
		4 & 4 & 5 & 0.8 \\
		5 & 2 & 5 & 0.6 \\ [1ex] 
		\hline
	\end{tabular}
\end{table}\\

 To solve this problem using the  homotopy method \ref{hybrid} with vector parameter $\lambda,$ we first take the initial point  $\tilde{x}^{(0)}=$$\left[\begin{array}{c} 
 1\\
 1\\
 1\\
 1\\
 1\\
 \end{array}\right],$ $\tilde{\lambda}^{(0)}=$$\left[\begin{array}{c} 
 1\\
 1\\
 1\\
 1\\
 1\\
 \end{array}\right].$ After $20$ iterations we obtain the result $\tilde{x}=$ $\left[\begin{array}{c} 
 15.429308\\
 12.498582\\
 9.663473\\
 7.165093\\
 5.132566\\
 \end{array}\right],$
 $\tilde{\lambda}=$$\left[\begin{array}{c}
 0\\
 0\\
 0\\
 0\\
 0\\
 \end{array}\right].$
 \section{Conclusion}
In this study we show the equivalency between the nonlinear complementarity problem and the system of nonlinear equations. We introduce a homotopy with vector parameter to find the solution of the system of nonlinear equations as well as the solution of the corresponding nonlinear complementarity problem. In this connection we obtain the sign of the positive tangent direction of the homotopy path. We construct a homotopy continuation method and show that the method reaches to solution through a bounded and smooth curve under the condition of the nonlinear function corresponding to the nonlinear complementarity problem being either increasing function or bounded function. Finally a real life oligopolistic market equilibrium problem is considered to illustrate our results.
\section{Acknowledgment}
The author A. Dutta is thankful to the Department of Science and Technology, Govt. of India, INSPIRE Fellowship Scheme for financial support.
\bibliographystyle{plain}
\bibliography{ref2}

\end{document}